\documentclass[review]{elsarticle}

\usepackage{lineno,hyperref}
\modulolinenumbers[5]
\journal{Journal of Scientific Computing}

\usepackage{amsmath,amssymb,amsthm,amsfonts,afterpage,float,bm,stmaryrd,paralist,wrapfig}

\usepackage{cancel}
\usepackage{color}
\usepackage{fancyhdr}
\usepackage{graphics}            
\usepackage{hyperref}  
\usepackage{graphicx,epsf}
\usepackage{makeidx}
\usepackage{epsfig}
\usepackage{lscape}

\pssilent

\newtheorem{theorem}{Theorem}[section]
\newtheorem{lemma}{Lemma}[section]

\newtheorem{remark}{Remark}[section]

\numberwithin{equation}{section}
\numberwithin{figure}{section}
\numberwithin{table}{section}

\def\XXint#1#2#3{{\setbox0=\hbox{$#1{#2#3}{\int}$}
\vcenter{\hbox{$#2#3$}}\kern-.51\wd0}}

\newcommand{\reff}[1]{{\rm (\ref{#1})}}

\graphicspath{{Figure/}}

\begin{document}

\setlength{\pdfpageheight}{\paperheight}
\setlength{\pdfpagewidth}{\paperwidth}
\title{Energy stable semi-implicit schemes for Allen-Cahn-Ohta-Kawasaki Model in Binary System}
\author{Xiang Xu}
\address{Department of Mathematics, Old Dominion University, Norfolk, VA 23529}
\author{Yanxiang Zhao\fnref{myfootnote}}
\address{Department of Mathematics, George Washington University, Washington D.C., 20052}
\fntext[myfootnote]{Corresponding author: yxzhao@email.gwu.edu}

\begin{abstract}
In this paper, we propose a first order energy stable linear semi-implicit method for solving the Allen-Cahn-Ohta-Kawasaki equation. By introducing a new nonlinear term in the Ohta-Kawasaki free energy functional, all the system forces in the dynamics are localized near the interfaces which results in the desire hyperbolic tangent profile. In our numerical method, the time discretization is done by some stabilization technique in which some extra nonlocal but linear term is introduced and treated explicitly together with other linear terms, while other nonlinear and nonlocal terms are treated implicitly. The spatial discretization is performed by the Fourier collocation method with FFT-based fast implementations. The energy stabilities are proved for this method in both semi-discretization and full discretization levels. Numerical experiments indicate the force localization and desire hyperbolic tangent profile due to the new nonlinear term. We test the first order temporal convergence rate of the proposed scheme. We also present hexagonal bubble assembly as one type of equilibrium for the Ohta-Kawasaki model. Additionally, the two-third law between the number of bubbles and the strength of long-range interaction is verified which agrees with the theoretical studies.
\end{abstract}


\date{\today}
\maketitle


\section{Introduction}\label{sec:Introduction}

Ohta-Kawasaki (OK) model introduced in  \cite{OhtaKawasaki_Macromolecules1986} has been used for the simulations of the phase separation of diblock copolymers, which are chain molecules consisting of two different and chemically incompatible segment species, say $A$ and $B$ species, connected by a covalent chemical bond. Diblock copolymers have generated much interest in materials science in recent years due to their remarkable ability for self-assembly into nanoscale ordered structures. This ability can be exploited to create materials with desired mechanical, optical electrical, and magnetic properties \cite{Hamley_Wiley2004}. The Ohta-Kawasaki model describes the binary system such as diblock copolymers by a free energy functional written as follows
\begin{align}\label{eqn:OKenergy}
E^{\text{OK}}[\phi] = \int_{\Omega} \left[ \dfrac{\epsilon}{2}|\nabla \phi|^2 + \dfrac{1}{\epsilon}W(\phi) \right] dx + \dfrac{\gamma}{2} \int_{\Omega} \Big|(-\Delta)^{-\frac{1}{2}}\Big(f(\phi)-\omega\Big)\Big|^2\ dx ,
\end{align}
where $0<\epsilon\ll 1$ is an interface parameter, $\Omega\subset \mathbb{R}^d, d = 2, 3$ is a spatial domain, and $\phi = \phi(x)$ is a phase field labeling function which represents the density of $A$ species. The concentration of $B$ species can be implicitly represented by $1-\phi(x)$ since the system is assumed to be incompressible \cite{OhtaKawasaki_Macromolecules1986}. $W(\phi) = 18 (\phi^2-\phi)^2$ is a double well potential which enforces the labeling function $\phi(x)$ to be equal to 0 or 1 in most of the domain $\Omega$ except the interfacial region. A new term of $f(\phi)= 6\phi^5-15\phi^4+10\phi^3$ is adapted to mimic $\phi$ as the indicator for the $A$ species. The first integral in (\ref{eqn:OKenergy}) is a local surface energy. The second term indicates the long-range interaction between the chain molecules with $\gamma>0$ being the strength of such an interaction. $\omega\in(0,1)$ is the relative volume of the $A$ species which indicates that the Ohta-Kawasaki model is usually associated with a volume constraint:
\begin{align}\label{eqn:vol_constraint}
\int_{\Omega} f(\phi)\ dx = \omega|\Omega|.
\end{align}
The negative square root of $-\Delta$ is defined in Section 2 in details. The  Allen-Cahn-Ohta-Kawasaki (ACOK) equation, the main focus in this paper, is resulted from the energy variation of the energy functional (\ref{eqn:OKenergy}) in the $L^2$ space coupled with the volume constraint (\ref{eqn:vol_constraint}). 

The main contribution of this paper lies in several aspects. Firstly, the new form of $f(\phi)$ will introduce a more localized boundary force near interfaces which will result in more desired tanh profile for the numerical solutions and consequently lead to more accurate free energy estimate. A more detailed discussion on the motivation of such a new $f(\phi)$ will be presented in Section 1.1. Secondly, due to the new nonlinear term $f(\phi)$ and the nonlocal interactions, we adapt a linear splitting method and the corresponding linear semi-implicit scheme for the $L^2$ gradient flow dynamics. Such a scheme is advantageous of treating all nonlocal nonlinear terms explicitly so that Fourier spectral method can be used to solve the fully discrete system efficiently when coupled with periodic boundary condition, and more importantly, it will inherit the energy dissipation law (energy stability) obeyed by the continuous $L^2$ gradient flow. Thirdly, our numerical simulations will validate the functionality of $f(\phi)$ in our model, and present the stability and efficiency of the proposed numerical schemes.

\subsection{Motivation of the new form $f(\phi)$}

The new form of $f(\phi)$
\begin{align}
f(\phi) = 6\phi^5-15\phi^4+10\phi^3
\end{align}
is due to several reasons. Firstly, heuristically, with such a function, we have not only $f(0)=0, f(1)=1$ which resembles the behavior of $\phi$, but also that
\begin{align}
f'(0) = 0\quad \text{and}\ f'(1) = 0.
\end{align}
These will lead to a more localized 'boundary force' near the $A$-$B$ interface (for example, see the last two terms of (\ref{eqn: OKgradient})). On the other hand, a more traditional choice $f(\phi) = \phi$ with $f'(\phi) = 1$ will induce a local surface tension force (the first two terms on the right hand side of (\ref{eqn: OKgradient})) against a global repulsive force (the last term on the right hand side of (\ref{eqn: OKgradient})). To balance such two forces, the $\phi$ might lose the desired tanh profile and results in either unphysical negative values in the proximity of the interface (see \cite{Zhao_2018CMS}) or values not equal 0 or 1 away from the interface (see Figure \ref{fig:phi_comparison}). Secondly, since the new $f(\phi)$ results in a much better tanh profile, consequently the new model will describe the interfacial structures more accurately and lead to a better estimate of the free energy. Thirdly, we can have a smooth linear extension of $f(\phi)$ at 0 and 1 as shown in the equation \reff{f_modified} so that the energy stability can be inherited later in Theorem \ref{theorem:semidiscrete_energystable} and Theorem \ref{theorem:fullydiscrete_energystable}.  Lastly, the choice of such a $f(\phi)$ is motivated by the following prototypical example, Ginzburg-Landau functional
\begin{align}\label{GinzburgLandau}
E^{GL}[\phi] = \int_{\Omega}\left[\dfrac{\epsilon}{2}|\nabla\phi|^2 + \dfrac{1}{\epsilon} W(\phi)\right] dx 
\end{align} 
subject to the volume constraint (\ref{eqn:vol_constraint}). The Euler-Lagrange equation of \reff{GinzburgLandau} and \reff{eqn:vol_constraint} reads:
\begin{align}\label{EulerLagrange}
-\epsilon\Delta\phi + \dfrac{1}{\epsilon}W'(\phi) + \lambda f'(\phi) = 0
\end{align}
where $\lambda$ is the Lagrange multiplier for the volume constraint. 

Following the work of \cite{DuLiuRyhamWang}, let us assume that $\phi_{\epsilon}$, a solution of the constrained minimization \reff{GinzburgLandau} and \reff{eqn:vol_constraint}, has an asymptotic expansion
\begin{align}\label{AsymExp}
\phi_{\epsilon} = q\left(\dfrac{d(x)}{\epsilon/3}\right) + \epsilon h\left(\dfrac{d(x)}{\epsilon/3}\right) + g,
\end{align}
where $q, h\in C^2(\Omega)$ are independent of $\epsilon$, $\|\nabla^k q\|_{L^{\infty}}=O(1)$, $\|\nabla^k h\|_{L^{\infty}}=O(1)$, $\|\nabla^k g\|_{L^{\infty}}= o(\epsilon)$ for $k=0,1,2$. The function $d(x) = \text{dist}(x,\Gamma_{\epsilon})$ with $\{\Gamma_{\epsilon}\}$ being a family of class $C^2$ compact surfaces converging uniformly to $\Gamma_0$. Applying the similar idea as in \cite{DuLiuRyhamWang}, it yields
\begin{align*}
\lambda = \lambda_0 + o(1).
\end{align*}
Furthermore, inserting the asymptotic expansion \reff{AsymExp} into \reff{EulerLagrange}, we can compare the $\epsilon^{-1}$ term to obtain 
\begin{align*}
q(\cdot) = 0.5 + 0.5 \tanh(\cdot),
\end{align*}
namely, up to $O(1)$, $\phi_{\epsilon}$ becomes 1 inside and 0 outside when away from the interface. 
For the $\epsilon^0$ term, we have
\begin{align}\label{epsilon0term}
-9 h''\left(\dfrac{d(x)}{\epsilon/3}\right) + W''\left(q\left(\dfrac{d(x)}{\epsilon/3}\right)\right) - &6\left(q\left(\dfrac{d(x)}{\epsilon/3}\right)-q\left(\dfrac{d(x)}{\epsilon/3}\right)^2\right)\Delta d(x) \nonumber \\
 + &30\left(q\left(\dfrac{d(x)}{\epsilon/3}\right)-q\left(\dfrac{d(x)}{\epsilon/3}\right)^2\right)^2\lambda_0 = 0.
\end{align}
When $x\in\Gamma_{\epsilon}$, we have $\Delta d(x) = 2H(x)$, and all other terms in \reff{epsilon0term} reduce to constants. Therefore, $H(x) \equiv H_{\epsilon}$ and the interface $\Gamma_{\epsilon}$ becomes a sphere. More importantly, when away from $\Gamma_{\epsilon}$, equation \reff{epsilon0term} reduces to $-9h'' + 36h=0$. Using the assumption of $h$, it leads to the solution $h=0$. Consequently, up to $O(\epsilon)$, $\phi_{\epsilon}$ remains as 1 inside and 0 outside when away from the interface $\Gamma_{\epsilon}$.

As a comparison, if we take $f(\phi) = \phi$, the $\epsilon^0$ equation reads
\begin{align*}
-9 h''\left(\dfrac{d(x)}{\epsilon/3}\right) + W''\left(q\left(\dfrac{d(x)}{\epsilon/3}\right)\right) - 6\left(q\left(\dfrac{d(x)}{\epsilon/3}\right)-q\left(\dfrac{d(x)}{\epsilon/3}\right)^2\right)\Delta d(x) + \lambda_0 = 0.
\end{align*}
When away from the interface, it reduces to $-9h'' + 36h + \lambda_0 = 0$ and the solution is $h =\lambda_0/36$. Therefore, up to $O(\epsilon)$, $\phi_{\epsilon}$ is not equal to 1 inside and 0 outside any more, but rather has a small deviation of $O(\epsilon \lambda_0/36)$. Such a deviation is also presented in the numerical simulation in Figure \ref{fig:phi_comparison}. Of course, the deviation of $\phi_{\epsilon}$ can be mitigated by letting $\epsilon\rightarrow0$ by the theoretical study (for instance as in \cite{LiZhao_SIAM2013}). However, in the real applications, especially in the 3D simulations, $\epsilon$ has to remain relatively large to reduce the computational cost. Therefore, the new $f(\phi)$ is advantageous for keeping the hyperbolic tangent profile of $\phi_{\epsilon}$ and localizing the forces only near the interfaces even for a relatively large $\epsilon$.

\begin{remark}
Using the penalty formulation, we can change the constrained minimization problem \reff{GinzburgLandau} and \reff{eqn:vol_constraint} into an unconstrained one: 
\begin{align}
\min_{\phi} E^{\text{GL}}_{M} [\phi] = E^{\text{GL}} [\phi] + \frac{M}{2}\left(\int_{\Omega} f(\phi) dx - \omega|\Omega|\right)^2
\end{align}
with $M\gg1$ being a penalty constant. Then solving the corresponding Euler-Lagrange equation at equilibrium:
\begin{align}\label{eqn:EL_penalty}
-\epsilon\Delta\phi + \dfrac{1}{\epsilon}W'(\phi) + M\left(\int_{\Omega} f(\phi) dx - \omega|\Omega|\right)f'(\phi) = 0
\end{align}
leads to an approximation of the Lagrange multiplier formulation (\ref{EulerLagrange}) with $\lambda \approx M\left(\int_{\Omega} f(\phi) dx - \omega|\Omega|\right)$, which becomes equivalent at the limit of $M\rightarrow \infty$. More importantly, $f'(\phi)$ in (\ref{eqn:EL_penalty}) localizes the volume force near the interface, which can be perfectly balanced by the tension force while maintaining the desired tanh profile. 
\end{remark}

\subsection{$L^2$ gradient flow dynamics}

Using the Lagrange multiplier approach, the $L^2$ gradient flow for the energy (\ref{eqn:OKenergy}) subject to the volume constraint (\ref{eqn:vol_constraint})  can be written as
\begin{align}\label{eqn: OKgradient}
\phi_t = - \dfrac{\delta E^{\text{OK}}[\phi]}{\delta \phi} = \epsilon\Delta\phi - \frac{1}{\epsilon}W'(\phi) - \gamma(-\Delta)^{-1}(f(\phi)-\omega) f'(\phi)  - \lambda(t) f'(\phi),
\end{align}
where $\lambda(t)$ is some appropriate time-dependent Lagrange multiplier such that the volume constraint (\ref{eqn:vol_constraint}) always holds.

Since the ACOK equation, as a $L^2$ gradient dynamics, is energy dissipative, it is natural to develop energy stable numerical scheme for it. To see how the volume constraint plays a role in the energy stability, we incorporate the penalty term into (\ref{eqn:OKenergy}) and change it into an unconstrained one:
\begin{align}\label{eqn:OKenergy_penalty}
E[\phi] = \int_{\Omega} \left[ \dfrac{\epsilon}{2}|\nabla \phi|^2 + \dfrac{1}{\epsilon}W(\phi) \right] dx + \dfrac{\gamma}{2} \int_{\Omega} \Big|(-\Delta)^{-\frac{1}{2}}\Big(f(\phi)-\omega\Big)\Big|^2 dx + \dfrac{M}{2} \left( \int_{\Omega}f(\phi) dx - \omega|\Omega| \right)^2,
\end{align}
and consider the corresponding penalized ACOK (pACOK) equation:
\begin{align}\label{eqn:OKgradientflow_penalized}
\dfrac{d\phi}{dt} = - \dfrac{\delta E[\phi]}{\delta\phi} =  \epsilon \Delta\phi - \dfrac{1}{\epsilon} W'(\phi) - \gamma (-\Delta)^{-1}(f(\phi)-\omega)f'(\phi) - M \left( \int_{\Omega}f(\phi) dx - \omega|\Omega| \right)f'(\phi) \ .
\end{align}
There has been extensive studies on the energy stable methods for various gradient flow dynamics. One example is the convex splitting method \cite{Eyre_Proc1998} in which the double well potential $W(\phi)$ is split into the sum of a convex function and a concave one, and the convex part is treated implicitly and the concave one is treated explicitly. However, a nonlinear system usually needs to be solved at each time step which induces high computational cost. Another widely adapted method is the stabilized semi-implicit method \cite{XuTang_SINA2006,ShenYang_DCDSA2010} in which $W(\phi)$ is treated explicitly. A linear stabilizing term is added to maintain the energy stability.  Another recent method is the IEQ method \cite{ChengYangShen_JCP2017, Yang_JCP2016} in which all nonlinear terms are treated semi-implicitly, the energy stability is preserved and the resulting numerical schemes lead to a symmetric positive definite linear system to be solved at each time step.

Due to the new nonlinear terms $f(\phi)$ and the nonlocal interactions, we will consider a linear splitting method for the OK energy functional and the resulting stabilized linear semi-implicit scheme for the pACOK equation such that all nonlocal nonlinear terms are treated explicitly. By doing so, we can solve the fully discrete system efficiently via Fourier spectral method. In this paper, we will firstly present our linear scheme and show the energy stability in semi-discretization level. Then we will discuss the energy stability for the full discretization with spectral discretization in space. Note that recently there has been several studies on the OK model for diblock copolymers. For example, \cite{Benesova_SINA2014} studies an implicit midpoint spectral approximation for the equilibrium of OK model. \cite{ChengYangShen_JCP2017} adopts the IEQ method to study the diblock copolymer model. However, the existing works mainly focus on the Cahn-Hilliard dynamics, namely, the $H^{-1}$ gradient flow dynamics of the OK energy (\ref{eqn:OKenergy}) with $f(\phi) = \phi$. Our work is advantageous of studying the force balance in the process of phase separation and emphasizing the force localization near the interface, and could have potential applications in other complex dynamics involving interfacial structures such as cell motility \cite{CamleyZhao_PRL2013} and implicit solvation \cite{Zhao_2018CMS}.

For the discussion of energy stability in Section 2 and Section 3, we revise the $W(s)$ quadratically and $f(s)$ linearly when $|s|>M_{\text{cut}}>0$ for some $M_{\text{cut}}$ in order to have finite upper bounds for $W''$ and $f''$. These modifications are necessary for the stabilized semi-implicit schemes for Ginzburg-Landau type dynamics \cite{ShenYang_DCDSA2010}. We will take the quadratic extension $\tilde{W}(\phi)$ of $W(\phi)$ as adapted in \cite{ShenYang_DCDSA2010} and other related citations therein. For the linear extension of $f(\phi)$, we choose:
\begin{align}\label{f_modified}
\tilde{f}(\phi) = 
\begin{cases}
0, & \phi<0,\\
f(\phi) = 6\phi^5 - 15 \phi^4 + 10\phi^3, & 0\le \phi \le 1, \\
1, & \phi>1.
\end{cases}
\end{align}
Such a linear extension will guarantee that $\tilde{f}(\phi)$ is Lipschitz continuous. We denote by $L_W, L_f$ the upper bounds of $|\tilde{W}''(s)|$ and $\tilde{f}''(s)$, respectively, and $L_p$ the Lipschitz constant for $\tilde{f}$. For brevity, we will still use $f(\phi), W(\phi)$ to represent $\tilde{f}(\phi), \tilde{W}(\phi)$, respectively.

The rest of the paper is organized as follows. In Section 2, we will introduce the penalty formulation of the OK model, then develop the first order (in time) stabilized linear semi-implicit scheme and analyze its energy stability in semi-discretization level. In Section 3, we will present the spectral discretization in space, and show the energy stability for the fully discretized scheme. In Section 4, we will give some numerical examples to illustrate the force localization, energy stability and the effect of the long-range repulsive force on the pattern formations.

\section{Ohta-Kawasaki Model in Penalty Formulation}

\subsection{Notations}

Let $\Omega = \prod_{i=1}^d [-X_i,X_i) \subset \mathbb{R}^d, d = 2,3$ be a periodic domain. Denote the space consisting of periodic functions in $H^s(\Omega), s\ge 0$ as $H^s_{\text{per}}(\Omega)$. We define the subspaces
\begin{align}
\mathring{H}^s_{\text{per}}(\Omega) := \bigg\{ u\in H^s_{\text{per}}(\Omega): \int_{\Omega} u(x) dx = 0 \bigg\}
\end{align}
consisting of all functions of $u\in H^s_{\text{per}}(\Omega)$ with zero mean. We use $\|\cdot\|_{H^s}$ to represent the standard Sobolev norm. When $s = 0$, $H^s(\Omega) = L^2(\Omega)$ and we take $\langle \cdot, \cdot \rangle$ as the $L^2$ inner product and $\|\cdot\|_{H^s} = \|\cdot\|_{L^2}$.

We define the inverse Laplacian $(-\Delta)^{-1}$: $\mathring{L}^2_{\text{per}}(\Omega)\rightarrow \mathring{H}^1_{\text{per}}(\Omega) $ as
\begin{align*}
(-\Delta)^{-1} g = u \Longleftrightarrow -\Delta u = g.
\end{align*}
or in term of Fourier series:
\begin{align}\label{eqn:inv_Lap}
 (-\Delta)^{-1}g = \sum_{k\in \mathbb{Z}^3\backslash\{0\}} |k|^{-2} \hat{g}(k)e^{\text{i} k\cdot \tilde{x}},
\end{align}
where 
\begin{align*}
\hat{g}(k) = \int_{\Omega}g(x)e^{-\text{i}k\cdot \tilde{x}}dx,\quad \text{with}\ \ \tilde{x} = (\pi x_1/X_1, \cdots, \pi x_d/X_d ).
\end{align*}
Note that the definition (\ref{eqn:inv_Lap}) can be extended to any function $g\in L_{\text{per}}^2(\Omega)$ because of the removal of the zero-th mode.

\subsection{Semi-discretization in time: a linear operator splitting method }

Now we will consider a semi-discrete scheme for the $L^2$ gradient flow dynamics of (\ref{eqn:OKenergy_penalty}), the pACOK equation (\ref{eqn:OKgradientflow_penalized}), and show that such a scheme will lead to energy stability under some constraints for the splitting coefficients. Given time interval $[0,T]$ and an integer $N>0$, we take the uniform time step size $\tau = \frac{T}{N}$ and $t_n = n\tau$ for $n=0,1,\cdots, N$.

Due to the nonlinearity and nonlocality in (\ref{eqn:OKenergy_penalty}) and (\ref{eqn:OKgradientflow_penalized}) , we consider a linear splitting of the energy functional (\ref{eqn:OKenergy_penalty}) as $E[\phi] = E_{l}[\phi] - E_{n}[\phi]$ with
\begin{align}\label{linear_splitting}
E_l[\phi] = &\int_{\Omega} \dfrac{\epsilon}{2}|\nabla \phi|^2 + \dfrac{\kappa}{2\epsilon}\phi^2  + \dfrac{\gamma}{2}\beta \left|(-\Delta)^{-\frac{1}{2}}\left(\phi-\omega\right)\right|^2 dx, \quad \text{and}\quad E_n[\phi] = E_l[\phi] - E[\phi].
\end{align}
Then treating $E_l$ implicitly and $E_n$ explicitly yields a first order linear semi-implicit scheme: for any $0\le n \le N-1$, find $\phi^{n+1}$ such that
\begin{align}\label{eqn:semi_implicit}
\dfrac{\phi^{n+1}-\phi^n}{\tau} = -\dfrac{\delta E_l}{\delta \phi} (\phi^{n+1}) + \dfrac{\delta E_n}{\delta \phi} (\phi^{n}),
\end{align}
with given initial data $\phi^0$, and splitting constants $\kappa$ and $\beta$. More precisely the scheme reads:
\begin{align}\label{eqn:semi_implicit2}
\dfrac{\phi^{n+1}-\phi^n}{\tau}  = &\ \epsilon \Delta \phi^{n+1} - \dfrac{\kappa}{\epsilon}\phi^{n+1} - \gamma\beta(-\Delta)^{-1}(\phi^{n+1}-\omega) + \dfrac{1}{\epsilon}\left[\kappa\phi^n-W'(\phi^n)\right] \nonumber \\
& + \gamma\left[\beta(-\Delta)^{-1}(\phi^n-\omega) - (-\Delta)^{-1}(f(\phi^n)-\omega)f'(\phi^n)\right] - M \left[\int_{\Omega} f(\phi^n)dx-\omega|\Omega|\right]f'(\phi^n).
\end{align}

We can rewrite the semi-discrete scheme (\ref{eqn:semi_implicit2}) as:
\begin{align}\label{eqn:semi_implicit3}
\left( \left(\dfrac{1}{\tau} + \dfrac{\kappa}{\epsilon} \right) I - \epsilon\Delta + \gamma\beta(-\Delta)^{-1}  \right)\phi^{n+1} = F^n 
\end{align}
with 
\begin{align*}
F^n =\ &\frac{\phi^n}{\tau} + \dfrac{1}{\epsilon}\left[\kappa\phi^n-W'(\phi^n)\right] 
+ \gamma\left[\beta(-\Delta)^{-1}(\phi^n-\omega) - (-\Delta)^{-1}(f(\phi^n)-\omega)f'(\phi^n)\right] \\
&- M \left[\int_{\Omega} f(\phi^n)dx-\omega|\Omega|\right]f'(\phi^n).
\end{align*}
Notice that $(-\Delta)^{-1}\phi^{n+1}$ in (\ref{eqn:semi_implicit3}) is well defined by removing the zero-th mode of $\phi^{n+1}$. A simple calculation reveals that all the eigenvalues of the operator on the left hand side of (\ref{eqn:semi_implicit3}) are positive. Therefore, the scheme is unconditionally uniquely solvable.

We show by the following theorem that the linear semi-implicit scheme (\ref{eqn:semi_implicit})  is energy stable, namely, it inherits the energy dissipation of (\ref{eqn:OKgradientflow_penalized}) as time increases. 
\begin{theorem}\label{theorem:semidiscrete_energystable}
The first-order convex splitting scheme (\ref{eqn:semi_implicit}) is energy stable: 
\begin{align}\label{eqn:energystability}
E[\phi^{n+1}] \le E[\phi^{n}],
\end{align}
provided that 
\begin{align}\label{eqn:kappa_beta}
\kappa \ge \frac{L_W}{2}  + \epsilon\left( \frac{\gamma L_f}{2}\|(-\Delta)^{-1}\|_{\infty}\max\{\omega,1-\omega \} + \frac{M}{2} |\Omega|\left( L_p^2 + L_f \max\{\omega,1-\omega \} \right) \right), \ \beta \ge \frac{L_p^2}{2}.
\end{align}
\end{theorem}

\begin{proof}
Let 
\begin{align*}
g(\phi) &= \dfrac{\kappa}{2}\phi^2 - W(\phi),\\
g'(\phi) & = \kappa\phi - W'(\phi),\\
h(\phi) &=  \beta|(-\Delta)^{-\frac{1}{2}}(\phi-\omega)|^2 - | (-\Delta)^{-\frac{1}{2}}(f(\phi)-\omega) |^2,\\
h'(\phi) & = \beta(-\Delta)^{-1}(\phi-\omega) - (-\Delta)^{-1}(f(\phi)-\omega) f'(\phi),\\
v(\phi) & = \int_{\Omega} f(\phi) dx - \omega|\Omega|.
\end{align*}
Note that $h'(\phi)$ is not really the derivative of $h(\phi)$, but rather the variational derivative of $\int h(\phi) dx$.

Taking the $L^2$ inner product of (\ref{eqn:semi_implicit2}) with $\phi^{n+1}-\phi^{n}$ leads to 
\begin{align}\label{eqn:estimate1}
&\dfrac{1}{\tau} \|\phi^{n+1} - \phi^{n}\|_{L^2}^2 \nonumber\\
= &\  \epsilon \langle \Delta\phi^{n+1}, \phi^{n+1} - \phi^{n} \rangle - \dfrac{\kappa}{\epsilon}\langle\phi^{n+1},\phi^{n+1} - \phi^{n}\rangle - \gamma\beta \langle (-\Delta)^{-1}(\phi^{n+1}-\omega), \phi^{n+1} - \phi^{n}\rangle \nonumber\\
& + \dfrac{1}{\epsilon}\langle g'(\phi^{n}), \phi^{n+1} - \phi^{n}\rangle + \gamma \langle h'(\phi^n),\phi^{n+1} - \phi^{n}\rangle - M v(\phi^n) \langle f'(\phi^n), \phi^{n+1} - \phi^{n}\rangle \nonumber \\
= &\ \text{I + II + III + IV + V + VI}.
\end{align}
Using the identity $a\cdot (a-b) = \frac{1}{2}|a|^2 - \frac{1}{2}|b|^2 + \frac{1}{2}|a-b|^2$, we have:
\begin{align*}
\text{I} =&\ -\dfrac{\epsilon}{2}\left( \|\nabla\phi^{n+1}\|_{L^2}^2 - \|\nabla\phi^{n}\|_{L^2}^2 + \|\nabla\phi^{n+1}-\nabla\phi^n\|_{L^2}^2 \right); \\
 \text{II} =&\ -\dfrac{\kappa}{2\epsilon} \left(  \|\phi^{n+1}\|_{L^2}^2 - \|\phi^{n}\|_{L^2}^2 + \|\phi^{n+1}-\phi^n\|_{L^2}^2  \right);\\
 \text{III} =&\ -\gamma\beta \langle (-\Delta)^{-1}(\phi^{n+1}-\omega), (\phi^{n+1}-\omega) - (\phi^{n}-\omega)\rangle \\
 \quad=&\ -\gamma\beta \langle (-\Delta)^{-\frac{1}{2}}(\phi^{n+1}-\omega), (-\Delta)^{-\frac{1}{2}}(\phi^{n+1}-\omega) - (-\Delta)^{-\frac{1}{2}}(\phi^{n}-\omega)\rangle\\
 \quad =&\ -\dfrac{\gamma\beta}{2}\left( \|(-\Delta)^{-\frac{1}{2}}(\phi^{n+1}-\omega)\|_{L^2}^2 - \|(-\Delta)^{-\frac{1}{2}}(\phi^{n}-\omega)\|_{L^2}^2 + \|(-\Delta)^{-\frac{1}{2}}(\phi^{n+1}-\phi^n)\|_{L^2}^2 \right); \\
 \text{IV}=&\  \frac{1}{\epsilon}\left(\langle g(\phi^{n+1}),1\rangle - \langle g(\phi^{n}),1\rangle - \frac{g''(\xi^n)}{2}\|\phi^{n+1}-\phi^n\|_{L^2}^2 \right) ,\quad \text{for some}\ \xi^n\ \text{between}\ \phi^n\ \text{and}\ \phi^{n+1}.
\end{align*}
For the term V, note that
\begin{align*}
\text{V} = &\ \gamma \langle h'(\phi^n),\phi^{n+1} - \phi^{n}\rangle \\
= &\ \gamma\langle \beta(-\Delta)^{-1}(\phi^n-\omega),\phi^{n+1}-\phi^n\rangle - \gamma \langle (-\Delta)^{-1}(f(\phi^n)-\omega)f'(\phi^n), \phi^{n+1} - \phi^n \rangle,
\end{align*}
then we have
\begin{align*}
&\gamma\langle \beta(-\Delta)^{-1}(\phi^n-\omega),\phi^{n+1}-\phi^n\rangle \\
=&\ \gamma\beta \langle (-\Delta)^{-1}(\phi^{n}-\omega), (\phi^{n+1}-\omega) - (\phi^{n}-\omega)\rangle \\
=&\ \gamma\beta \langle (-\Delta)^{-\frac{1}{2}}(\phi^{n}-\omega), (-\Delta)^{-\frac{1}{2}}(\phi^{n+1}-\omega) - (-\Delta)^{-\frac{1}{2}}(\phi^{n}-\omega)\rangle\\
=&\ \dfrac{\gamma\beta}{2}\left( \|(-\Delta)^{-\frac{1}{2}}(\phi^{n+1}-\omega)\|_{L^2}^2 - \|(-\Delta)^{-\frac{1}{2}}(\phi^{n}-\omega)\|_{L^2}^2 - \|(-\Delta)^{-\frac{1}{2}}(\phi^{n+1}-\phi^n)\|_{L^2}^2 \right),
\end{align*}
and 
\begin{align*}
&-\gamma \left\langle (-\Delta)^{-1}(f(\phi^n)-\omega)f'(\phi^n), \phi^{n+1} - \phi^n \right\rangle \\
=&\ -\gamma \left\langle (-\Delta)^{-1}(f(\phi^{n})-\omega), f'(\phi^n)\phi^{n+1} - \phi^n\right\rangle \\
=&\ -\gamma \left\langle (-\Delta)^{-1}(f(\phi^{n})-\omega), f(\phi^{n+1}) - f(\phi^n) - \frac{f''(\eta^n)}{2}(\phi^{n+1}-\phi^n)^2\right\rangle \\
=&\ -\frac{\gamma}{2}\left( \|(-\Delta)^{-\frac{1}{2}}(f(\phi^{n+1})-\omega)\|_{L^2}^2 - \|(-\Delta)^{-\frac{1}{2}}(f(\phi^{n})-\omega)\|_{L^2}^2 -\|(-\Delta)^{-\frac{1}{2}}(f(\phi^{n+1})-f(\phi^n))\|_{L^2}^2 \right)\\
&\ + \gamma\left\langle (-\Delta)^{-1}(f(\phi^{n})-\omega), \frac{f''(\eta^n)}{2}(\phi^{n+1}-\phi^n)^2\right\rangle, \quad \text{for some}\ \eta^n\ \text{between}\ \phi^n\ \text{and}\ \phi^{n+1},
\end{align*}
where we use the identity $b\cdot (a-b) = \frac{1}{2}|a|^2 - \frac{1}{2}|b|^2 - \frac{1}{2}|a-b|^2$. For the term VI, we have
\begin{align*}
\text{VI} & = -Mv(\phi^n) \langle f'(\phi^n)(\phi^{n+1} - \phi^{n}),1\rangle \\
& = -Mv(\phi^n)\Big\langle (f(\phi^{n+1}) - \omega) - (f(\phi^n)-\omega) -\frac{f''(\eta^n)}{2}(\phi^{n+1}-\phi^n),1 \Big\rangle \\
& = -M v(\phi^n) (v(\phi^{n+1}) - v(\phi^n)) + \frac{M}{2} v(\phi^n) f''(\eta^n)\|\phi^{n+1}-\phi^n\|_{L^2}^2 \\
& = -\frac{M}{2} \left( |v(\phi^{n+1})|^2 - |v(\phi^{n})|^2 -|v(\phi^{n+1})-v(\phi^n)|^2 \right) + \frac{M}{2} v(\phi^n) f''(\eta^n)\|\phi^{n+1}-\phi^n\|_{L^2}^2.
\end{align*}
Inserting the equalities for I -VI back into (\ref{eqn:estimate1}) yields
\begin{align*}
&\dfrac{1}{\tau}\|\phi^{n+1}-\phi^n\|_{L^2}^2 + \frac{\epsilon}{2}\|\nabla\phi^{n+1}-\nabla\phi^n\|_{L^2}^2 + E[\phi^{n+1}] - E[\phi^n] \\
=& -\frac{\kappa}{2\epsilon}\|\phi^{n+1}-\phi^n\|_{L^2}^2 - \dfrac{\gamma\beta}{2}\|(-\Delta)^{-\frac{1}{2}}(\phi^{n+1}-\phi^n)\|_{L^2}^2 - \frac{g''(\xi^n)}{2\epsilon}\|\phi^{n+1}-\phi^n\|_{L^2}^2 \\
& - \dfrac{\gamma\beta}{2}\|(-\Delta)^{-\frac{1}{2}}(\phi^{n+1}-\phi^n)\|_{L^2}^2 + \frac{\gamma}{2}\|(-\Delta)^{-\frac{1}{2}}(f(\phi^{n+1})-f(\phi^n))\|_{L^2}^2 \\
& + \frac{\gamma}{2}\left\langle (-\Delta)^{-1}(f(\phi^{n})-\omega), f''(\eta^n)(\phi^{n+1}-\phi^n)^2\right\rangle + \frac{M}{2} \left(  \left| v(\phi^{n+1})-v(\phi^n)\right|^2 +  v(\phi^n) f''(\eta^n)\|\phi^{n+1}-\phi^n\|_{L^2}^2 \right)
\end{align*}
Note that $g''(\xi^n) = \kappa - W''(\xi^n)$, we further have
\begin{align*}
&\dfrac{1}{\tau}\|\phi^{n+1}-\phi^n\|_{L^2}^2 + \frac{\epsilon}{2}\|\nabla\phi^{n+1}-\nabla\phi^n\|_{L^2}^2 + E[\phi^{n+1}] - E[\phi^n] \\
=& -\frac{\kappa}{\epsilon}\|\phi^{n+1}-\phi^n\|_{L^2}^2 - \gamma\beta\|(-\Delta)^{-\frac{1}{2}}(\phi^{n+1}-\phi^n)\|_{L^2}^2 + \frac{W''(\xi^n)}{2\epsilon}\|\phi^{n+1}-\phi^n\|_{L^2}^2 \\
& + \frac{\gamma}{2}\|(-\Delta)^{-\frac{1}{2}}(f(\phi^{n+1})-f(\phi^n))\|_{L^2}^2 + \frac{\gamma}{2}\int_{\Omega} (-\Delta)^{-1}(f(\phi^{n})-\omega) f''(\eta^n)(\phi^{n+1}-\phi^n)^2 dx \\
&  + \frac{M}{2} \left| v(\phi^{n+1})-v(\phi^n)\right|^2 + \frac{M}{2} v(\phi^n) f''(\eta^n)\|\phi^{n+1}-\phi^n\|_{L^2}^2 \\
\le& -\frac{\kappa}{\epsilon}\|\phi^{n+1}-\phi^n\|_{L^2}^2 - \gamma\beta\|(-\Delta)^{-\frac{1}{2}}(\phi^{n+1}-\phi^n)\|_{L^2}^2 + \frac{L_W}{2\epsilon}\|\phi^{n+1}-\phi^n\|_{L^2}^2 + \frac{\gamma}{2}L_p^2 \|(-\Delta)^{-\frac{1}{2}}(\phi^{n+1}-\phi^n)\|_{L^2}^2 \\
& + \frac{\gamma}{2}L_f\|(-\Delta)^{-1}\|_{L^{\infty}}\|f(\phi^n)-\omega\|_{L^{\infty}} \|\phi^{n+1}-\phi^n\|_{L^2}^2 + \frac{M}{2} |\Omega|\left( L_p^2 + L_f \max\{\omega,1-\omega \} \right) \|\phi^{n+1}-\phi^n\|_{L^2}^2 \\
= & \left(-\frac{\kappa}{\epsilon} +\frac{L_W}{2\epsilon} + \frac{\gamma L_f}{2}\|(-\Delta)^{-1}\|_{L^{\infty}}\max\{\omega,1-\omega \} + \frac{M}{2} |\Omega|\left( L_p^2 + L_f \max\{\omega,1-\omega \} \right)  \right)  \|\phi^{n+1}-\phi^n\|_{L^2}^2 \\
& + \left(-\gamma\beta  +\frac{\gamma}{2}L_p^2\right) \|(-\Delta)^{-\frac{1}{2}}(\phi^{n+1}-\phi^n)\|_{L^2}^2 \le 0,
\end{align*}
where the last inequality is true provided (\ref{eqn:kappa_beta}) holds. Consequently it leads to the energy stability.
\end{proof}

\begin{remark}
Note that the linear splitting of the energy functional in \reff{linear_splitting} is not a convex splitting. More specifically, $E_l[\phi]$ is convex, but $E_n[\phi]$ is non-convex due to the nonlinearity of $f(\phi)$. Therefore we cannot prove the energy stability as it is done via Lemma 3.1 in \cite{JuLiQiaoZhang_MathComput2017} and Theorem 1.1 in \cite{WiseWangLowengrub_SINA2009}. On the other hand, we can still show the discrete energy dissipation law by noticing that the nonlinear function $f(\phi)$ in \reff{f_modified} is bounded in $|f''|$ and Lipschitz continuous, which is the key in the proof.
\end{remark}

\section{Energy stability for the fully discrete scheme}

In this section, we will use spectral approximation in space to construct the fully discrete scheme and analyze its energy stability as a discrete analogy of (\ref{eqn:energystability}).

\subsection{Spectral collocation approximation for spatial discretization}\label{sec:spectral_collocation}

We discretize the spatial operators by using the spectral collocation approximation. To this end, we adopt some notations for the spectral approximation as in \cite{JuLiQiaoZhang_MathComput2017, DuJuLiQiao_JCP2018}.

We consider $\Omega = [-X,X)\times[-Y,Y)\subset \mathbb{R}^2$. For the three-dimensional case, the notations can be defined in similar manner.  Let $N_x$ and $N_y$ be two positive even integers. We take $h_x = \frac{2X}{N_x}$ and $h_y = \frac{2Y}{N_y}$ and define $\Omega_h = \Omega\cap(h_x\mathbb{Z}\otimes h_y\mathbb{Z})$. We define the index sets:
\begin{align*}
S_h &= \left\{ (i,j)\in\mathbb{Z}^2 | 1\le i \le N_x, 1\le j \le N_y \right\}, \\
\hat{S}_h &= \left\{ (k,l)\in\mathbb{Z}^2 | -\frac{N_x}{2}+1\le k \le \frac{N_x}{2}, -\frac{N_y}{2}+1\le j \le \frac{N_y}{2} \right\}.
\end{align*}
Denote by $\mathcal{M}_h$ the collection of periodic grid functions on $\Omega_h$:
\begin{align*}
\mathcal{M}_h = \left\{ f: \Omega_h\rightarrow\Omega | f_{i+mN_x, j+nN_y} =f_{ij}, \forall (i,j)\in S_h, \forall(m,n)\in \mathbb{Z}^2 \right\}.
\end{align*}
For any $f,g\in\mathcal{M}_h$ and $\textbf{f} = (f^1,f^2)^T, \textbf{g} = (g^1,g^2)^T\in\mathcal{M}_h\times\mathcal{M}_h$, we define the discrete $L^2$ inner product $\langle\cdot,\cdot\rangle_h$ and discrete $L^2$ norm $\|\cdot\|_{h,L^2}$ and discrete $L^{\infty}$ norm $\|\cdot\|_{h, L^{\infty}}$ as follows:
\begin{align*}
\langle f, g\rangle_h &= h_x h_y \sum_{(i,j)\in S_h} f_{ij} g_{ij}, \quad \|f\|_{h,L^2} = \sqrt{\langle f, f \rangle_h}, \quad \|f\|_{h, L^{\infty}} = \max_{(i,j)\in S_h} |f_{ij}| ; \\
\langle \textbf{f}, \textbf{g}\rangle_h &= h_x h_y \sum_{(i,j)\in S_h} \left( f_{ij}^1 g_{ij}^1 + f_{ij}^2 g_{ij}^2 \right), \quad  \|\textbf{f}\|_{h,L^2} = \sqrt{\langle \textbf{f}, \textbf{f}\rangle_h}.
\end{align*}

For a function $f\in\mathcal{M}_h$, the 2D discrete Fourier transform (DFT) $\hat{f} = Pf$ is defined as:
\begin{align}
\hat{f}_{kl} =  \frac{1}{N_xN_y} \sum_{(i,j)\in S_h} f_{ij} \exp\left( -\text{i} \frac{k\pi}{X}x_i \right) \exp\left( -\text{i} \frac{l\pi}{Y}y_j \right), \quad (k,l)\in\hat{S}_h,
\end{align}
where `i' is the complex unity, and $x_i = -X + ih_x, y_j = -Y + jh_y, 1\le i \le N_x, 1\le j \le N_y$. The corresponding inverse DTF (iDFT) is given as:
\begin{align}\label{eqn:iDFT}
f_{ij} =  \sum_{(k,l)\in \hat{S}_h} \hat{f}_{kl} \exp\left( \text{i} \frac{k\pi}{X}x_i \right) \exp\left( \text{i} \frac{l\pi}{Y}y_j \right), \quad (i,j)\in S_h.
\end{align}

Let $\widehat{\mathcal{M}}_h = \{Pf | f\in\mathcal{M}_h\}$ and define the operators $\hat{D}_x, \hat{D}_y$ on $\widehat{\mathcal{M}}_h$ as
\begin{align}
(\hat{D}_x \hat{f})_{kl} = \left(  \frac{\text{i}k\pi}{X} \right) \hat{f}_{kl}, \quad (\hat{D}_y \hat{f})_{kl} = \left(  \frac{\text{i}l\pi}{Y} \right) \hat{f}_{kl},\quad (k,l)\in\hat{S}_h,
\end{align}
then the Fourier spectral approximations to the the spatial operators $\partial_x, \partial_{xx}$ can be written as
\begin{align*}
D_x = P^{-1} \hat{D}_x P,\quad D_y = P^{-1} \hat{D}_y P, \quad D_x^2 = P^{-1} \hat{D}_x^2 P, \quad D_y^2 = P^{-1} \hat{D}_y^2 P.
\end{align*}
For any $f\in\mathcal{M}_h$ and $\textbf{f} = (f^1,f^2)^T\in\mathcal{M}_h\times\mathcal{M}_h$, the discrete gradient, divergence and Laplace operators are given respectively by
\begin{align*}
\nabla_h f = (D_x f, D_y f)^T, \quad \nabla_h\cdot\text{f} = D_x f^1 + D_y f^2,\quad \Delta_h f = D_x^2 f + D_y^2 f = P^{-1}(\hat{D}_x^2 + \hat{D}_y^2) P f.
\end{align*}
Let $\mathring{\mathcal{M}}_h = \{f\in\mathcal{M}_h | \langle f, 1 \rangle_h = 0\}$ be the collections of all periodic grid functions with zero mean, we define $(-\Delta_h)^{-1}: \mathring{\mathcal{M}}_h\rightarrow \mathring{\mathcal{M}}_h$ as
\begin{align*}
(-\Delta_h)^{-1} f = u \Longleftrightarrow -\Delta_h u = f.
\end{align*}
More precisely in terms of DFT and iDFT, we define it as
\begin{align*}
(-\Delta_h)^{-1} f = -P^{-1}(\hat{D}_x^2 + \hat{D}_y^2)^{-1} P f = -P^{-1}
\begin{cases}
\left[\left( \frac{k\pi}{X} \right)^2 + \left( \frac{l\pi}{Y} \right)^2\right]^{-1}\hat{f}_{kl}, & (k,l)\ne (0,0) \\
0, & (k,l)= (0,0)
\end{cases}
\end{align*}
Note that $(-\Delta_h)^{-1}$ can be defined for $g\in\mathcal{M}_h\backslash \mathring{\mathcal{M}}_h$ by simply removing the $(0,0)$-mode $\hat{g}_{00}$, namely, $(-\Delta_h)^{-1} g : = (-\Delta_h)^{-1} (g-\hat{g}_{00})$ if $g\in\mathcal{M}_h\backslash \mathring{\mathcal{M}}_h$. Furthermore, we define $(-\Delta_h)^{-\frac{1}{2}}$ as follows
\begin{align*}
\|(-\Delta_h)^{-\frac{1}{2}} f\|_{h,L^2}^2 = \big\langle (-\Delta_h)^{-\frac{1}{2}} f , (-\Delta_h)^{-\frac{1}{2}} f \big\rangle_h = \left\langle (-\Delta_h)^{-1} f , f \right\rangle_h
\end{align*}

It is easy to verify the following discrete integration by parts formulas:
\begin{lemma}\label{lemma:discrete_formulas}
For any functions $f,g\in\mathcal{M}_h$ and $ \textbf{g} = (g^1,g^2)^T\in\mathcal{M}_h\times\mathcal{M}_h$, we have
\begin{align*}
\langle f, \nabla_h\cdot\mathbf{g}\rangle_h = - \langle \nabla_h f, \mathbf{g} \rangle_h, \quad \langle f, \Delta_h g\rangle_h = -\langle \nabla_h f , \nabla_h g \rangle_h = \langle \Delta_h f , g \rangle_h.
\end{align*}
\end{lemma}

Using Lemma \ref{lemma:discrete_formulas} and the definition of $(-\Delta_h)^{-1}$, we can easily verify that 
\begin{align*}
(-\Delta_h)(-\Delta_h)^{-1} (f-\hat{f}_{00}) =(-\Delta_h)^{-1}(-\Delta_h) f = f - \hat{f}_{00},\quad \forall f\in\mathcal{M}_h.
\end{align*}

We further define the discrete $H^s$ norms for $f\in\mathcal{M}_h$:
\begin{align}
\|f\|_{h, H^s}^2 = \sum_{(k,l)\in \hat{S}_h} \Big( 1+(k^2+l^2)^s \Big) |\hat{f}_{kl}|^2.
\end{align}

In what follows, we will provide an upper $L^{\infty}$ bound of the operator $(-\Delta_h)^{-1}$ to guarantee the energy stability for the fully discrete scheme \reff{eqn:full_discrete}. To this end, we need two lemmas, the first one is the discrete analogy of a Solobev embedding inequality, the second one is about the stability of the spectrally discrete Laplacian operator $-\Delta_h$. Note from the following that the first Lemma is a consequence of the standard Sobolev embedding theory, and the second Lemma is a simple consequence of the elliptic regularity, we will not repeat the detailed proofs in here.  Readers who are interested in the proofs can refer to Sobolev space books, for instance \cite{Adams_SobolevSpace}.

\begin{lemma}\label{lemma:Solobev}
For any functions $f \in\mathcal{M}_h$, we have
\begin{align}\label{eqn:discrete_Solobev_embedding}
\|f\|_{h,L^{\infty}} \le C_s \|f\|_{h,H^s}
\end{align}
provided $s > d/2, d = 2, 3$, where $C_s$ is a constant only depending on $s$ and independent of $h$.
\end{lemma}

%

\begin{remark}
Taking $s = 2$ in the Lemma \ref{lemma:Solobev}, we can have an estimate on $C_2$:
\begin{align}
C_2^2 = \sum_{(k,l)\in \mathbb{Z}^2} \dfrac{1}{1+(k^2+l^2)^2} \le 1 + 4\sum_{k=1}^{\infty}\frac{1}{1+k^2} + \int_{\mathbb{R}^2 }  \frac{1}{1+(x^2+y^2)^2} dxdy \le 1 + 4\cdot\frac{\pi^2}{6} + \dfrac{\pi^2}{2}.
\end{align}
\end{remark}

We also have an uniform $L^{\infty}$ bound for the operator $(-\Delta_h)^{-1}$.
\begin{lemma}\label{lemma:invLap}
Let any $u, f \in\mathcal{M}_h$ be such that $-\Delta_h u = f$, then we have
\begin{align}\label{eqn:invLap_uniform_bound}
\|u\|_{h,L^{\infty}} \le C_2\sqrt{(1+C_p^4)|\Omega|}\ \|f\|_{h, L^\infty},
\end{align}
where $C_p \le \frac{\max\{X,Y\}}{\pi}$ is the discrete Poincare constant. In other words, $\|(-\Delta_h)^{-1}\|_{h, L^{\infty}}$ is uniformly bounded.
\end{lemma}

%

\subsection{Energy stability for fully discrete semi-implicit scheme}

In this section, we use the spectral collocation approximations defined in Section \ref{sec:spectral_collocation} for the spatial discretization and construct the first order fully discrete semi-implicit scheme for the pAOCK equation (\ref{eqn:OKgradientflow_penalized}).

The first order fully discrete semi-implicit scheme for the pACOK equation (\ref{eqn:OKgradientflow_penalized}) reads: for $0\le n \le N-1$, find $\phi_h^{n+1} = (\phi_{ij}^{n+1})\in \mathcal{M}_h$ such that
\begin{align}\label{eqn:full_discrete}
\dfrac{\phi_h^{n+1}-\phi_h^n}{\tau}  = &\ \epsilon \Delta_h \phi_h^{n+1} - \dfrac{\kappa_h}{\epsilon}\phi_h^{n+1} - \gamma\beta_h(-\Delta_h)^{-1}(\phi_h^{n+1}-\omega) + \dfrac{1}{\epsilon}\left[\kappa_h\phi_h^n-W'(\phi_h^n)\right] \nonumber \\
& + \gamma\left[\beta_h(-\Delta_h)^{-1}(\phi_h^n-\omega) - (-\Delta_h)^{-1}(f(\phi_h^n)-\omega)f'(\phi_h^n)\right] - M \left[ \langle f(\phi_h^{n}),1\rangle_h - \omega|\Omega| \right] f'(\phi_h^n).
\end{align}
with $\phi_h^0 = (\phi_{ij}^0)\in\mathcal{M}_h$ being the given initial data, $\kappa_h$ and $\beta_h$ two stabilization constants.

It is similar as in (\ref{eqn:semi_implicit3}) to show the unconditional unique solvability for (\ref{eqn:full_discrete}). To this end, we rewrite the scheme as:
\begin{align}\label{eqn:full_discrete2}
\left( \left(\dfrac{1}{\tau} + \dfrac{\kappa_h}{\epsilon} \right) I - \epsilon\Delta_h + \gamma\beta_h(-\Delta_h)^{-1}  \right)\phi_h^{n+1} = F_h^n 
\end{align}
with 
\begin{align*}
F_h^n =\ &\frac{\phi_h^n}{\tau} + \dfrac{1}{\epsilon}\left[\kappa_h\phi_h^n-W'(\phi_h^n)\right] 
+ \gamma\left[\beta_h(-\Delta_h)^{-1}(\phi_h^n-\omega) - (-\Delta_h)^{-1}(f(\phi_h^n)-\omega)f'(\phi_h^n)\right] \\
&- M \left[ \langle f(\phi_h^{n}),1\rangle_h - \omega|\Omega| \right] f'(\phi_h^n).
\end{align*}
Then all the eigenvalues for the operator on the left hand side of (\ref{eqn:full_discrete2}) are
\begin{align}
\lambda_{kl} = 
\begin{cases}
\frac{1}{\tau} + \frac{\kappa_h}{\epsilon} + \epsilon\left(\left( \frac{k\pi}{X} \right)^2 + \left( \frac{l\pi}{Y} \right)^2\right) + \gamma\beta_h \left(\left( \frac{k\pi}{X} \right)^2 + \left( \frac{l\pi}{Y} \right)^2\right)^{-1}, & (k,l)\ne (0,0),\\
\frac{1}{\tau} + \frac{\kappa_h}{\epsilon}, & (k,l) = (0,0).
\end{cases}
\end{align}
Consequently, $\lambda_{kl}>0$, which implies the unconditional unique solvability of the fully discrete scheme (\ref{eqn:full_discrete}).

To consider the energy stability for the scheme (\ref{eqn:full_discrete}), let us define a discrete analogy of the energy (\ref{eqn:OKenergy_penalty}):
\begin{align}
E_h[\phi_h] = \frac{\epsilon}{2}\|\nabla_h\phi_h\|_{h,L^2}^2 + \frac{1}{\epsilon}\langle W(\phi_h), 1 \rangle_h + \frac{\gamma}{2}\|(-\Delta_h)^{-\frac{1}{2}}(f(\phi_h-\omega))\|_{h,L^2}^2 + \frac{M}{2}\left( \langle f(\phi_h), 1 \rangle_h - \omega|\Omega| \right)^2.
\end{align}

Then we have the following energy stability for the scheme  (\ref{eqn:full_discrete}):

\begin{theorem}\label{theorem:fullydiscrete_energystable}
For any $\tau>0$, the $\phi_h^{n+1}$ determined by the scheme  (\ref{eqn:full_discrete}) satisfies:
\begin{align}\label{eqn:energy_stability_discrete}
E_h[\phi_h^{n+1}] \le E_h[\phi_h^n],
\end{align}
provided that the constants $\kappa_h$ and $\beta_h$ satisfy
\begin{align}\label{eqn:kappah_betah}
\kappa_h \ge \frac{L_W}{2}  + \epsilon\left( \frac{\gamma L_f}{2}\|(-\Delta_h)^{-1}\|_{h, L^{\infty}}\max\{\omega,1-\omega \} + \frac{M}{2} |\Omega|\left( L_p^2 + L_f \max\{\omega,1-\omega \} \right) \right), \ \beta_h \ge \frac{L_p^2}{2}.
\end{align}
\end{theorem} 

\begin{proof}

The proof is similar to the one for Theorem \ref{theorem:semidiscrete_energystable}. To see how the discrete operators work in the proof, we will still present the techniques in detail.

Take the discrete $L^2$ inner product with $\phi_h^{n+1} - \phi_h^{n}$ on the two sides of (\ref{eqn:full_discrete}), we have
\begin{align}\label{eqn:full_discrete_1}
\frac{1}{\tau}\| \phi_h^{n+1} - \phi_h^{n} \|_{h,L^2}^2  = &\ \epsilon \left \langle \Delta_h \phi_h^{n+1} , \phi_h^{n+1} - \phi_h^{n} \right \rangle_h - \frac{\kappa_h}{\epsilon} \left \| \phi_h^{n+1} - \phi_h^{n}  \right \|_{h,L^2}^2 - \frac{1}{\epsilon}\left \langle W'(\phi_h^{n}), \phi_h^{n+1} - \phi_h^{n} \right \rangle_h  \nonumber \\
& - \gamma\beta_h \left \langle (-\Delta_h)^{-1}(\phi_h^{n+1}-\phi_h^n),\phi_h^{n+1}-\phi_h^n \right \rangle_h  - \gamma \left \langle (-\Delta_h)^{-1}(f(\phi_h^{n})-\omega)f'(\phi_h^n),\phi_h^{n+1} - \phi_h^{n} \right \rangle_h \nonumber \\
& - M \big( \left \langle f(\phi_h^{n}),1 \right \rangle_h - \omega|\Omega|\ \big) \left \langle f'(\phi_h^n), \phi_h^{n+1} - \phi_h^{n} \right \rangle_h \nonumber \\
= & \text{I} + \text{II} + \text{III} + \text{IV} + \text{V} + \text{VI}.
\end{align}
Using the identities $a(a-b)= \frac{1}{2}|a|^2 - \frac{1}{2}|b|^2 + \frac{1}{2}|a-b|^2$ and $b(a-b)= \frac{1}{2}|a|^2 - \frac{1}{2}|b|^2 - \frac{1}{2}|a-b|^2$ yields
\begin{align*}
\text{I} & = -\frac{\epsilon}{2}\left( \|\nabla_h\phi_h^{n+1}\|_{h,L^2}^2 - \|\nabla_h\phi_h^{n}\|_{h,L^2}^2 + \|\nabla_h\phi_h^{n+1}-\nabla_h\phi_h^{n}\|_{h,L^2}^2 \right); \\
\text{II} & = - \frac{\kappa_h}{\epsilon}\| \phi_h^{n+1} - \phi_h^{n} \|_{h,L^2}^2; \\
\text{III} & = -\frac{1}{\epsilon}\left\langle W(\phi_h^{n+1}),1\right\rangle_h + \frac{1}{\epsilon}\left\langle W(\phi_h^{n}),1\right\rangle_h + \frac{1}{2\epsilon}\left\langle W''(\xi_h^{n}), (\phi_h^{n+1}-\phi_h^n)^2 \right\rangle_h,\ \text{for some}\ \xi^n_h\ \text{between}\ \phi^n_h\ \text{and}\ \phi^{n+1}_h; \\
\text{IV} & = - \gamma\beta_h\left\| (-\Delta_h)^{-\frac{1}{2}}(\phi_h^{n+1}-\phi_h^n) \right\|_{h,L^2}^2; \\
\text{V} & = -\gamma \left\langle (-\Delta_h)^{-1}(f(\phi_h^{n})-\omega), f'(\phi_h^n)(\phi_h^{n+1} - \phi_h^{n})\right\rangle_h \\
& = - \gamma \left\langle (-\Delta_h)^{-1}(f(\phi_h^{n})-\omega), (f(\phi_h^{n+1}) - \omega) - (f(\phi_h^{n}) - \omega)\right\rangle_h \\ 
&\quad\ + \dfrac{\gamma}{2} \left\langle (-\Delta_h)^{-1}(f(\phi_h^{n})-\omega), f''(\eta^n_h)(\phi_h^{n+1}-\phi_h^n)^2 \right\rangle_h \\
& = -\frac{\gamma}{2}\left( \| (-\Delta_h)^{-\frac{1}{2}}(f(\phi_h^{n+1})-\omega) \|_{h,L^2}^2 -  \| (-\Delta_h)^{-\frac{1}{2}}(f(\phi_h^{n})-\omega) \|_{h,L^2}^2 -  \| (-\Delta_h)^{-\frac{1}{2}}(f(\phi_h^{n+1})-f(\phi_h^{n})) \|_{h,L^2}^2  \right) \\
& \quad\  + \dfrac{\gamma}{2} \left\langle (-\Delta_h)^{-1}(f(\phi_h^{n})-\omega)f''(\eta^n_h), (\phi_h^{n+1}-\phi_h^n)^2 \right\rangle_h, \ \text{for some}\ \eta^n_h\ \text{between}\ \phi^n_h\ \text{and}\ \phi^{n+1}_h; \\
\text{VI} & = -M \big( \langle f(\phi_h^{n}),1\rangle_h - \omega|\Omega|  \big) \left( \left\langle f(\phi_h^{n+1}), 1 \right\rangle_h - \omega|\Omega| - \left\langle f(\phi_h^{n}), 1 \right\rangle_h + \omega|\Omega|   \right) \\
&\quad\   + \dfrac{M}{2} \big( \langle f(\phi_h^{n}),1\rangle_h - \omega|\Omega| \big) \left\langle f''(\eta^n_h), (\phi_h^{n+1}-\phi_h^n)^2 \right\rangle_h \\
& = -\frac{M}{2}\left( \left( \langle f(\phi_h^{n+1}),1\rangle_h - \omega|\Omega| \right)^2 - \left( \langle f(\phi_h^{n}),1\rangle_h - \omega|\Omega| \right)^2 - \left( \langle f(\phi_h^{n+1})-f(\phi_h^n),1\rangle_h  \right)^2  \right) \\
& \quad\ + \dfrac{M}{2} \big( \langle f(\phi_h^{n}),1\rangle_h - \omega|\Omega| \big)  \left\langle f''(\eta^n_h), (\phi_h^{n+1}-\phi_h^n)^2 \right\rangle_h.
\end{align*}
Inserting the above equations for I-VI back into (\ref{eqn:full_discrete_1}) yields
\begin{align*}
&\frac{1}{\tau}\| \phi_h^{n+1} - \phi_h^{n} \|_{h,L^2}^2 + \frac{\epsilon}{2}\|\nabla_h\phi_h^{n+1}-\nabla_h\phi_h^{n}\|_{h,L^2}^2 + E_h[\phi_h^{n+1}] - E_h[\phi_h^n] \\
=& - \frac{\kappa_h}{\epsilon}\| \phi_h^{n+1} - \phi_h^{n} \|_{h,L^2}^2  - \gamma\beta_h \| (-\Delta_h)^{-\frac{1}{2}}(\phi_h^{n+1}-\phi_h^n) \|_{h,L^2}^2 + \frac{\gamma}{2}\| (-\Delta_h)^{-\frac{1}{2}}(f(\phi_h^{n+1})-f(\phi_h^{n})) \|_{h,L^2}^2 \\
&+ \frac{1}{2\epsilon}\left\langle W''(\xi^n_h) , (\phi_h^{n+1}-\phi_h^n)^2 \right\rangle_h + \dfrac{\gamma}{2} \left\langle (-\Delta_h)^{-1}(f(\phi_h^{n})-\omega)f''(\eta^n_h), (\phi_h^{n+1}-\phi_h^n)^2 \right\rangle_h \\
& + \frac{M}{2}\left( \langle f(\phi_h^{n+1})-f(\phi_h^n),1\rangle_h  \right)^2 + \dfrac{M}{2} \big( \langle f(\phi_h^{n}),1\rangle_h - \omega|\Omega| \big)  \langle f''(\eta^n_h), (\phi_h^{n+1}-\phi_h^n)^2\rangle_h \\
\le & \bigg\langle - \frac{\kappa_h}{\epsilon} + \frac{L_W}{2\epsilon}  + \dfrac{\gamma L_f}{2} \|(-\Delta_h)^{-1}\|_{h,L^{\infty}} \max\{\omega,1-\omega\} + \frac{M}{2}|\Omega| \left( L_p^2 + L_f \max\{\omega,1-\omega\} \right), ( \phi_h^{n+1} - \phi_h^{n} )^2 \bigg\rangle_h \\
& + \left( -\gamma\beta_h + \frac{\gamma}{2}L_p^2 \right) \| (-\Delta_h)^{-\frac{1}{2}}(\phi_h^{n+1}-\phi_h^n) \|_{h,L^2}^2 \le 0
\end{align*}
where the last inequality holds provided (\ref{eqn:kappah_betah}) holds. Therefore we obtain the discrete energy stability (\ref{eqn:energy_stability_discrete}).
\end{proof}

\begin{remark}
Note that $\beta$ and $\beta_h$ satisfy the same stability condition as shown in (\ref{eqn:kappa_beta}) and (\ref{eqn:kappah_betah}). The key difference between (\ref{eqn:kappa_beta}) and (\ref{eqn:kappah_betah}) is the restriction on $\kappa$ and $\kappa_h$ which is exclusively determined by the $L^{\infty}$ norm of the inverse Laplacian and its spectrally discrete counterpart. Therefore applying the concrete bound on $\|(-\Delta_h)^{-1}\|_{h,L^{\infty}}$ provided in Lemma \ref{lemma:invLap} leads to an estimate on the value of $\kappa_h$.
\end{remark}

\section{Numerical Experiments}

In this section, we will use our first order semi-implicit scheme (\ref{eqn:full_discrete}) to solve the pACOK equation (\ref{eqn:OKgradientflow_penalized}) coupled with periodic boundary condition. In this section, we fix $\Omega = [-1,1]^2\subset\mathbb{R}^2$, $N=N_x = N_y = 512$. Then we have $|\Omega| = 4$ and  $h=h_x = h_y = \frac{1}{256}$. We set the stopping criteria for our time iteration by
\begin{align}\label{eqn:stopping}
\dfrac{\|\phi^{n+1} - \phi^n \|_{h, L^{\infty}}}{\tau} \le \text{TOL} = 10^{-3}.
\end{align}
The penalty constant is taken as $M = 1000$. By knowing the estimates on  $L_p, L_W, L_f$ and $\|(-\Delta_h)^{-1}\|_{h,L^{\infty}}$, we take stabilization constants $\kappa_h=2000$ and $\beta_h = 2$ to guarantee the discrete energy stability for $\epsilon = 5h, 10h, 20h$. The volume fraction $\omega=0.15$ is fixed in all numerical experiments. Other parameters such as $\epsilon$, $\gamma$ and $\tau$ vary for different simulations.

\subsection{$f(\phi)$ maintains tanh profile }

As the first example, we take $\epsilon = 20h, \gamma = 100$ and the initial condition $\phi^0$ as the characteristic function of a disc $\{(x,y): x^2 + y^2 \le r_0^2\}$ in which $r_0 = \sqrt{\omega|\Omega|/\pi}$. Because of the radial symmetry, the solution $\phi^{n}$ will remain being radially symmetric. The time step size is taken uniformly as $\tau = 10^{-3}$. 

In this example, we compare the cases in which $f(\phi) = \phi$ (old model) and $f(\phi) = 6\phi^5 - 15\phi^4 + 10\phi^3$ (new model). Note that for the old model, the stabilized dynamics (\ref{eqn:full_discrete}) reduces to a dynamics with $f(\phi) = \phi$ and $\beta_h = 1$.

In Figure \ref{fig:phi_comparison}, we present the solutions $\phi$ and the corresponding force distributions along the cross section of $y = 0$. our new phase field implementation improves the old one in several aspects. (1) The new model displays a better hyperbolic tangent profile than the old one as seen in Figure \ref{fig:phi_comparison}(a). More specifically, the equilibrium phase field $\phi$ in the new model shows a desirable hyperbolic tangent shape which monotonically changes its value from 1 to 0, while the old model presents some unphysical feature away the interfacial region, where $\phi$ has deviations of $O(10^{-2})$ from 0 (inside the interface) and 1 (outside the interface). (2) The new model maintains the force localization near the interface as seen in Figure \ref{fig:phi_comparison}(b). In the old model, all the three forces have nonzero contributions everywhere in the domain. Of course, the derivation of $\phi$ can be mitigated by letting $\epsilon\rightarrow 0$. However, in real applications, especially in the high dimensional simulations ($d=2$ or $d=3$), $\epsilon$ has to remain relatively large to reduce the computational cost. Therefore, the new model is advantageous for keeping the hyperbolic tangent profile of $\phi$ and localizing the forces only near the interfaces even for a relatively large $\epsilon$. It is worth mentioning that the force localization due to $f(\phi) = 3\phi^2 - 2\phi^3$ occurs not only at the equilibrium, but in the entire gradient-flow dynamics. Therefore it can potentially be used to study non-equilibrium dynamics such as cell motion \cite{CamleyZhao_PRL2013,CamleyZhao_PRE2017}.

\begin{figure}[htbp]
\centerline{
\includegraphics[width=200mm]{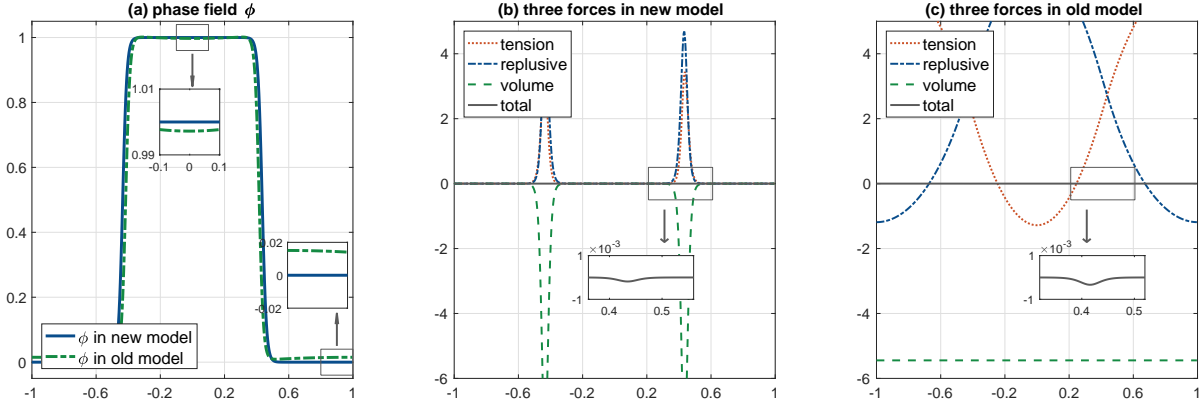}
$ \quad $ }
\vspace{-4 mm}
\caption{Numerical comparison between the new model and the old one. (a) The two phase-field functions $\phi$ at equilibrium in which the $\phi$ of new model presents a desirable hyperbolic tangent profile, but the $\phi$ of old model displays a deviation of $O(10^{-2})$ from 0 and 1 away from the interface as seen in the insets.  (b) The three forces in the new model (surface tension, repulsive force from nonlocal $(-\Delta)^{-1}$ term, and the volume force) are localized only near the interface and the sum is balanced up to $O(10^{-3})$.  (c) The three forces in the old model make nonzero contributions in the entire region. All the three subfigures are plotted by taking the cross section  at $y=0$ for the phase field $\phi$. In this simulation, $\gamma = 100, \epsilon = 20h$ and $\tau = 10^{-3}$.}
\label{fig:phi_comparison}
\end{figure}

We further test the convergence rates of the scheme (\ref{eqn:full_discrete}). To this end, we take a radially symmetric initial data $\phi^0 = 0.5+ 0.5 \tanh(\frac{r_0-r}{\epsilon/3})$ with  $r_0 = \sqrt{\omega|\Omega|/\pi}+0.1$. We perform the simulation until $T = 0.1$. We take the solution generated by the scheme (\ref{eqn:full_discrete}) with $\tau = 10^{-5}$ as the benchmark solution and then compute the discrete $L^2$ error between the numerical solutions with larger step sizes and benchmark one. Table \ref{table:rate} presents the errors and the convergence rates based on the data at $T = 0.1$ for the scheme  (\ref{eqn:full_discrete}) with time step sizes being halved from $\tau = 10^{-1}$ to $1.5625\times 10^{-3}$. We test the convergence rates for three different values of $\epsilon = 20h, 10h$ and $5h$. $\gamma = 100$ is fixed. We can see from the table that the numerically computed convergence rates all tend to approach the theoretical value 1. 

\begin{table}[H]
\begin{center}
\begin{tabular}{lllllllllllllr}
\hline
$\Delta t$ &  \multicolumn{2}{l}{$\epsilon = 20 h$} & & \multicolumn{2}{l}{$\epsilon = 10 h$} & &  \multicolumn{2}{l}{$\epsilon = 5 h$} \\
\cline{2-3} \cline{5-6} \cline{8-9}
    & Error & Rate & &  Error & Rate & & Error & Rate \\
\hline
$1.0000\times 10^{-1}$           & 1.9791   & --     & &  2.0670 & --      & &  2.3193 & -- \\
$5.0000\times 10^{-2}$           & 1.3012   & 0.605 & &  1.4223 & 0.540 & &  1.6309 & 0.508 \\
$2.5000\times 10^{-2}$          & 0.8255   & 0.657 & &  0.9547 & 0.575 & &  1.1408 & 0.516\\
$1.2500\times 10^{-2}$          & 0.5146   & 0.682 & &  0.6134 & 0.638 & &  0.7905 & 0.529 \\
$6.2500\times 10^{-3}$          & 0.3217   & 0.678 & &  0.3680 & 0.737 & &  0.5395 & 0.551 \\
$3.1250\times 10^{-3}$          & 0.1907   & 0.755 & & 0.2017 & 0.867 & &  0.3611 & 0.579 \\
$1.5625\times 10^{-3}$          & 0.0920   & 1.052 & & 0.0981 & 1.040 & &  0.1902 & 0.925\\
$10^{-5}$ (Benchmark)        & --   & -- & &  -- & -- & &   -- & -- \\
\hline
\end{tabular}\label{table:rate}
\caption{The errors and the corresponding convergence rates at time $T=0.1$ by the scheme  (\ref{eqn:full_discrete}) for different values of $\epsilon$. In this simulation, $\gamma = 100$ is fixed.}
\end{center}
\end{table}

\subsection{Coarsening dynamics}

\begin{figure}[htbp]
\centerline{
\includegraphics[width=150mm]{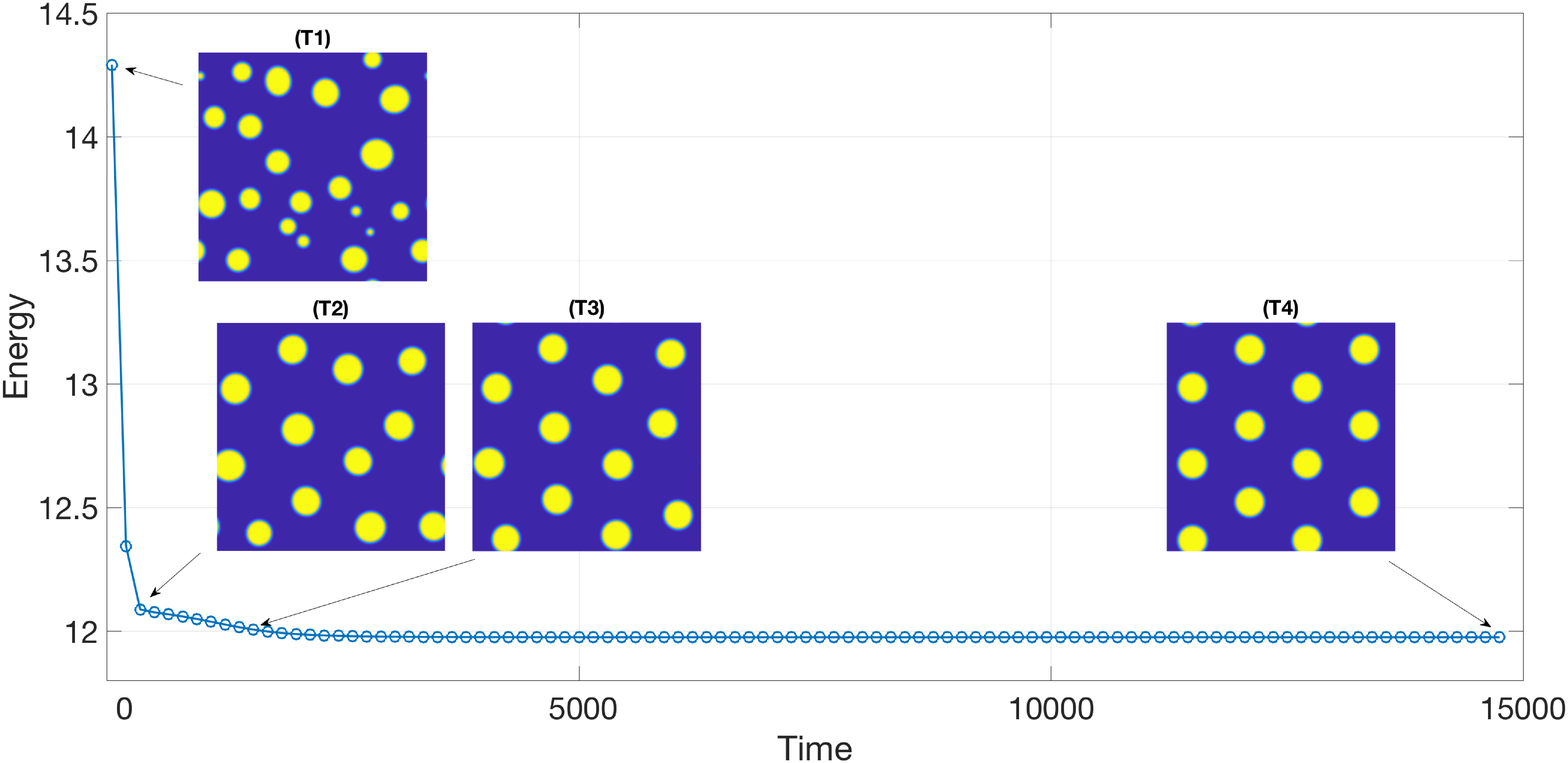}
 }
 \centerline{
\includegraphics[width=150mm]{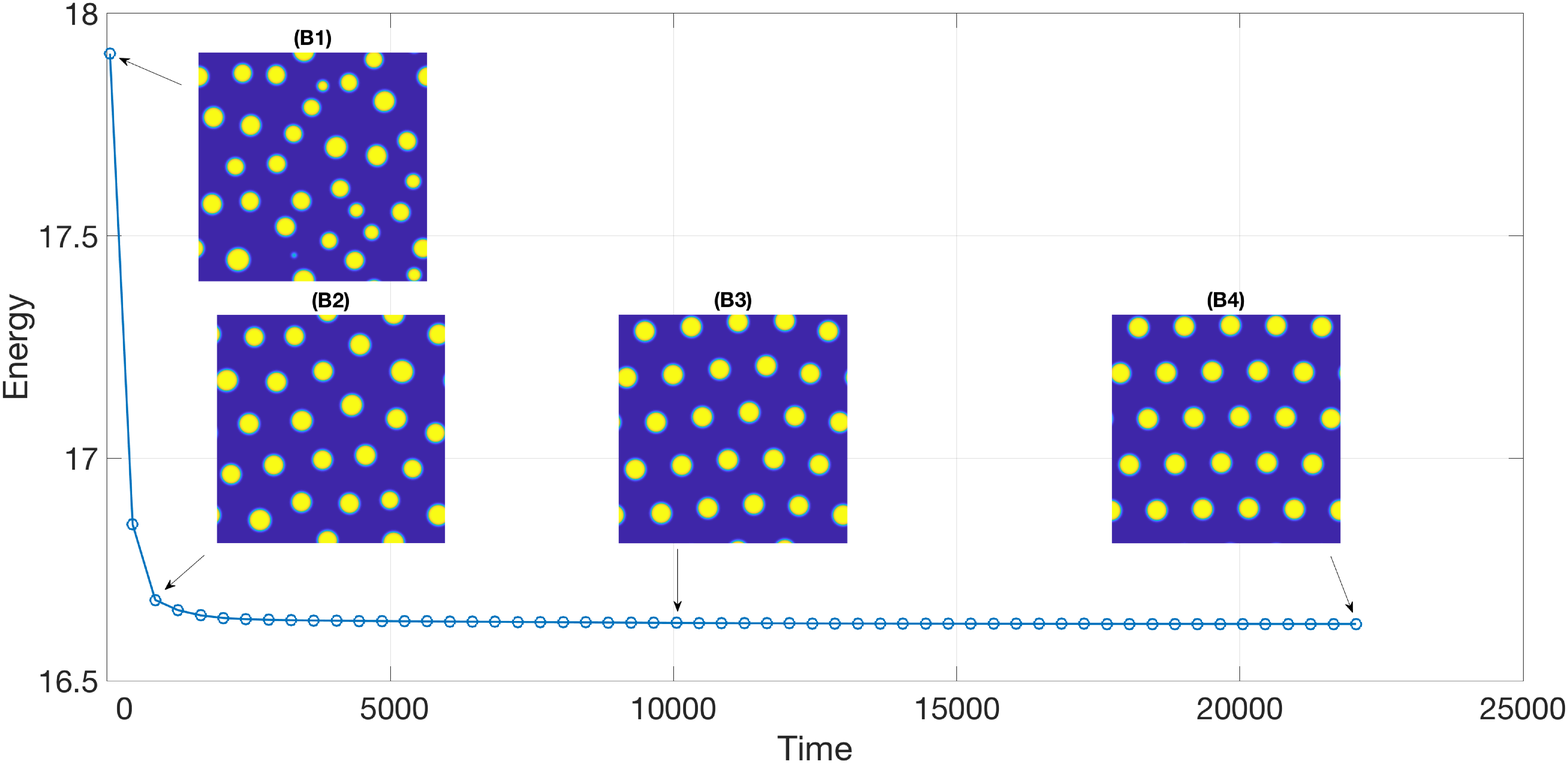}
 }
\caption{Two coarsening dynamic processes show the system experiences phase separation from a random initial, then bubbles appear from coarsening, evolve into same size, and finally form hexagonal structure. (Top): $\gamma = 2000$; (Bottom): $\gamma = 5000$. Other parameter values are $\tau = 5\times 10^{-3}$, $\omega = 0.15$, $\epsilon = 10h$.}
\label{fig:coarsening}
\end{figure}

In the second example, we consider the dynamics of the pACOK equation using the scheme (\ref{eqn:full_discrete}). When taking $\omega\ll 1$ and a relatively large $\gamma$, the pACOK equation results in an equilibrium of so-called bubble assembly in which one phase ($A$-species) is embedded into the other one ($B$-species).

The initial configuration is a random state given by random numbers varying uniformly in [0,1] on a coarse uniform mesh of $\Omega = [-1,1]^2$ with mesh size being $16h$. In Matlab, this random state can be easily generated by \texttt{repelem(rand(N/ratio,N/ratio),ratio,ratio)} in which the \texttt{ratio} stands for the ratio of the mesh size between the coarse and fine meshes. We take $\tau = 5\times 10^{-3}$ and terminate the simulation when the stopping condition (\ref{eqn:stopping}) is satisfied. We have tested that for different values of \texttt{ratio}, as long as \texttt{ratio}$\cdot h$ is of the size no larger than the bubble size, pACOK equation leads to bubble assemblies with the same number of bubbles. On the other hand, the larger the  \texttt{ratio} is, the faster the coarsening dynamics is. To capture the coarsening dynamics yet expedite the simulation, we choose \texttt{ratio}=16. 

Figure \ref{fig:coarsening} presents two pACOK coarsening dynamics with $\gamma = 2000$ and $5000$ respectively. For each subfigure, we start from a random state generated by \texttt{repelem}. In a very short time period, the random state is coarsen and bubbles of different size appear as shown in the inset (T1) ((B1), respectively). Then small bubbles disappear from (T1) to (T2) ( (B1) to (B2), respectively). Interestingly the free energy at (T2) ((B2), respectively) is the spot with the largest curvature along the energy curve. This is a critical spot in the energy decay. Before this spot, the coarsening dynamics drives the small bubbles merging with large ones, then all bubbles grow into the same size, and the number of bubbles become fixed at this spot. After this spot, bubbles of same size move around and eventually form a hexagonal structure as shown from (T3) to (T4) ((B3) to (B4), respectively). In addition, the energy curve clearly shows that the proposed linear semi-implicit scheme (\ref{eqn:full_discrete}) inherits the feature of energy dissipation.


\subsection{$\gamma$ controls number of bubbles}

\begin{figure}[htbp]
\centerline{
\includegraphics[width=140mm]{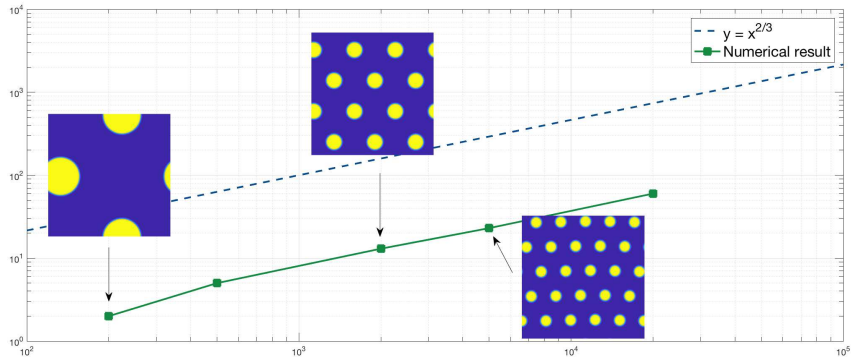}
$ \quad $ }
\vspace{-4 mm}
\caption{Log-log plot of the dependence of the number of bubbles on $\gamma$ in the bubble assemblies. Here $\omega = 0.15$. As $\gamma$ increases, the number of bubbles $K_b$ in the assemblies grows accordingly. For $\gamma = 200,500,2000,5000,20000$, the corresponding number of double bubbles are $K_b =  2, 5, 13, 23, 60$, respectively. The three insets are equilibrium states for $\gamma = 200, 2000, 20000$. The number of bubbles obeys the 2/3 law $K_b \sim \gamma^{2/3}$. In this simulation, $\epsilon = 10h$ and $\tau = 5\times 10^{-3}$. }
\label{fig:bubble_number}
\end{figure}

The parameter $\gamma$ plays a critical role in the bubble assembly. $\gamma$ measures the strength of the repulsion (long-range interaction) between bubbles and favors small domains, therefore, the larger $\gamma$ is, the more bubbles the assembly has. In the third example, we fix $\omega = 0.15$ but take a few values of $\gamma = 200, 500, 2000, 5000, 20000$.  In Figure \ref{fig:bubble_number} we see that as $\gamma$ increases, the number of bubbles $K_b$ change as $K_b = 2, 5, 13, 23, 60$. In order to diminish the error of bubble counting caused by the random state, we run five simulations for each value of $\gamma$, count number of bubbles and take average. We find that the number of bubbles remains the same regardless of different initial random states. More importantly, the relation between the number of bubbles $K_b$ and the repulsion strength $\gamma$ obeys the 2/3-law $K_b \sim \gamma^{2/3}$ which agrees with the theoretical studies \cite{RenWei_RMP2007}.

\section{Concluding remarks}

In this paper, we study a first order stabilized linear semi-immplicit scheme for the pACOK equation. by the introduction of a new term $f(\phi)$ in the OK free energy functional, we can localize the forces near the interface and maintain the solution as the desire hyperbolic tangent profile in the entire pACOK dynamics. We prove the energy stability in the semi-discrete and fully discrete schemes in which the stabilizers $\kappa_h$ and $\beta_h$ depend on the bounds of the second order derivatives of the nonlinear terms $W$ and $f$ as well as the Lipschitz continuity of $f$, or equivalently depend on the discrete $L^{\infty}$ bound of the numerical solutions.

In the numerical simulations, we validate the functionality of $f(\phi)$ on the force localization as we proposed in Section \ref{sec:Introduction}. When the volume fraction of $A$ species is much smaller than that of $B$ species, we find out that the pACOK dynamics leads to the pattern of hexagonal bubble assemblies. More important, we numerically verify that the number of bubbles in the hexagonal patterns is determined by the repulsion strength $\gamma$ through a two third law, which is consistent with the recent theoretical study.

We are currently working on a more general system of $N+1$ constituents which is described by a free energy functional:
\begin{align}\label{Energy_Ncomp}
E^N[\phi_1, \cdots, \phi_N] = &\int_{\Omega} \dfrac{\epsilon}{2} \sum_{\substack{i,j=0\\ i\le j}}^{N} \nabla\phi_i\cdot\nabla\phi_j + \dfrac{1}{2\epsilon} \left[ \sum_{i=1}^N W(\phi_i) + W\Big(1-\sum_{i=1}^N \phi_i\Big)  \right] dx \nonumber \\
&\quad + \sum_{i,j = 1}^N \frac{\gamma_{ij}}{2}  \int_{\Omega} \left[ \left(-\triangle\right)^{-\frac{1}{2}}  \left( f(\phi_{i}) -\omega_{i}\right) \left(-\triangle\right)^{-\frac{1}{2}}\left( f(\phi_{j}) -\omega_{j}\right) \right] dx,
\end{align}
where $\omega_i = \frac{1}{|\Omega|}\int_{\Omega} f(\phi_i) dx, i = 1,\cdots, N$ represent the volume constraints for each constituent $\phi_i$. When $N=2$, equation (\ref{Energy_Ncomp}) agrees with the ternary system which has been studied recently in \cite{NakazawaOhta_Macromolecules1993,RenWei_PhysD2003,RenWei_ARMA2015,WangRenZhao_CMS2018}. We try to incorporate the energy stable scheme into this system to better describe the bubble assemblies and possibly explore other interesting patterns. 

In this paper, our main focus is the first order stabilized linear semi-implicit scheme. In the future, we will conduct the error estimates for the fully discrete schemes. Moreover, we will investigate some higher order schemes and perform energy stability analysis if possible. Some other numerical methods, such as exponential time differencing based schemes, coupled with Fourier spectral discretization in space will also be under our consideration in the future as it is an efficient and stable numerical method which has been successfully applied to other gradient flow type dynamics \cite{JuLiQiaoZhang_MathComput2017, WangJuDu_JCP2016, Zhao_2018CMS}.

\section{Acknowledgements}

Y. Zhao's work is supported by Columbian College Facilitating Funds (CCFF 2018) of George Washington University and a grant from the Simons Foundation through Grant No. 357963.

\newpage

\section*{References}

\bibliography{OhtaKawasaki}

\end{document}